\documentclass[12pt]{amsart}
\usepackage[utf8]{inputenc}
\usepackage{amsmath} 
\usepackage{amssymb}
\usepackage{graphicx}
\usepackage{graphicx,color}
\usepackage{latexsym}
\usepackage[all]{xy}
\usepackage{mathrsfs}
\usepackage{enumerate}
\usepackage{amssymb}
\usepackage{bm}
\usepackage{tikz}
\usepackage{tikz-qtree}
\usepackage{breqn}
\usepackage{easyReview}
\usepackage{subfig}
\usepackage{hyperref}
\usetikzlibrary{shapes.geometric}

\makeatletter
\renewcommand{\subsection}{\@startsection
{subsection}{2}{0mm}{\baselineskip}{-0.25cm}
{\normalfont\normalsize\em}}
\makeatother

\def\neg1{\text{\boldmath$1$}}

\def\neg1{\text{\boldmath$1$}}

\def\neg0{\bf {0}}

\def\cM{\mathcal M}

\def\ZZ{\mathbb{Z}}

\newcommand{\fx}{\mathbf{x}}

\newtheorem{theorem}{Theorem}[section]
\newtheorem{proposition}[theorem]{Proposition}
\newtheorem{corollary}[theorem]{Corollary}
\newtheorem{lemma}[theorem]{Lemma}

{\theoremstyle{definition}
\newtheorem{definition}[theorem]{Definition}
\newtheorem{example}[theorem]{Example}

\newtheorem{conjecture}[theorem]{Conjecture}}

{\theoremstyle{remark}
\newtheorem{remark}[theorem]{Remark}
}

\makeatletter
\@namedef{subjclassname@2020}{\textup{2020} Mathematics Subject Classification}
\makeatother

\title[On some classes of generalized numerical semigroups]{On some classes of generalized numerical semigroups}

\author[C. Cisto]{Carmelo Cisto}
 \address{Università di Messina, Dipartimento di Scienze Matematiche e Informatiche, Scienze Fisiche e Scienze della Terra, Viale Ferdinando Stagno D’Alcontres 31, 98166 Messina, Italy}
  \email{carmelo.cisto@unime.it}

  \author[F. Navarra]{Francesco Navarra}
 \address{Università di Messina, Dipartimento di Scienze Matematiche e Informatiche, Scienze Fisiche e Scienze della Terra, Viale Ferdinando Stagno D’Alcontres 31, 98166 Messina, Italy}
  \email{francesco.navarra@unime.it}

\begin{document}


\begin{abstract}
 A generalized numerical semigroup is a submonoid of $\mathbb{N}^d$ with finite complement in it. In this work we study some properties of three different classes of generalized numerical semigroups. In particular, we prove that the first class satisfies a generalization of Wilf's conjecture, by introducing a generalization of a well-known sufficient condition for Wilf's conjecture in numerical semigroups, that involves the type of the semigroup. Partial results for Wilf's generalized conjecture are obtained also for the other two classes, and some open questions are provided.
 
\end{abstract}

\maketitle

\section{introduction}

Let $\mathbb{N}$ be the set of non negative integers and $d\geq 1$ be an integer. A monoid $S$ contained in $\mathbb{N}^d$ is called a \emph{generalized numerical semigroups} if the set $\mathbb{N}^d\setminus S$ is finite. This notion is introduced in \cite{failla2016algorithms} as a generalization of the well known definition of a numerical semigroup, that is a submonoid of $\mathbb{N}$ having finite complement in it. Numerical semigroups are studied in several papers and constitute an active area of research. For a collection of the fundamental concepts and the main results related to this matter, refer to the monographs \cite{numericalApplications, rosales2009numerical}. The introduction of generalized numerical semigroups leads to the natural goal of studying, in the new context, some concepts related to numerical semigroups. This is the intent of \cite{failla2016algorithms}, where the authors also provide some definitions and procedures obtained extending some arguments known for numerical semigroups. Successively, other papers on generalized numerical semigroups have appeared, introducing new properties and results. We will recall, along this paper, some of these properties and results that are useful for the purpose of this work. For other recent developments in such a matter see also \cite{bernardini2022corner,cisto2021algorithms,cisto2021almost}.
One of the main questions provided in \cite{failla2016algorithms} is to generalize a well known conjecture on numerical semigroups, called \emph{Wilf's conjecture}, introduced for the first time in \cite{wilf1978circle}. Although such a conjecture is proved to be true for many classes of numerical semigroups, it is still an open problem to prove it in its full generality. See \cite{delgado2019conjecture} for an exhaustive survey on this argument. Addressed to the previous question, a generalization of Wilf's conjecture for generalized numerical semigroups is introduced in \cite{cisto2020generalization}, referred as \emph{generalized Wilf's conjecure} (see also \cite{garcia2018extension} for a different extension of Wilf's conjecture in a more general context), and it is proved to be true for some particular classes of generalized numerical semigroups. So a natural direction of research is to investigate other classes of generalized numerical semigroups, studying their main properties and verifying the generalized Wilf's conjecture for them.\\ In such a direction in this paper we introduce and study some new classes of generalized numerical semigroups. In Section 2 we recall the generalized Wilf's conjecture and the arguments related to it, together with all useful concepts for the rest of this paper. In particular, a sufficient condition is introduced for a generalized numerical semigroup in order to verify the generalized Wilf's conjecture, involving a particular invariant called the \emph{type}. In Section 3, we introduce the concept of $T$-stripe generalized numerical semigroup, that is a semigroup in $\mathbb{N}^d$, $d\geq 2$, related to a fixed numerical semigroup $T$. We prove some properties of it depending on the fixed numerical semigroup $T$ and, using the sufficient condition introduced in Section 2, we prove that all generalized numerical semigroups of such a class satisfy the generalized Wilf's conjecture. Other two classes of generalized numerical semigroups are introduced in Section 4 and Section 5, proving some properties of them and verifying the generalized Wilf's conjecture in some particular cases. The two classes of semigroups introduced in Section 3 and Section 4 are inspired, in a certain sense, by some semigroups in $\mathbb{N}^2$ depicted in \cite{matthews2001weierstrass}. We conclude with some remarks and open questions, one of them provided by Shalom Elihaou, after a personal communication with him.

\section{preliminaries}

Recall that a numerical semigroup $S$ is a submonoid of $\mathbb{N}$ such that $\mathbb{N}\setminus S$ is a finite set. We denote by $\operatorname{m}(S)=\min(S\setminus \{0\}$ the \emph{multiplicity} of $S$, and $\operatorname{F}(S)=\max(\mathbb{Z}\setminus S)$ the Frobenius number of $S$. We consider in this paper a straightforward generalization of the concept of a numerical semigroup, provided for the first time in \cite{failla2016algorithms}, named \emph{generalized numerical semigroup} (GNS for short), that is a submonoid of $\mathbb{N}^d$ having finite complement in it. In particular all definitions we introduce for GNSs can be trivially considered also for numerical semigroups. \\
So, let $S\subseteq \mathbb{N}^d$ be a GNS. We consider the following notations:

\begin{itemize}
\item $\operatorname{H}(S)=\mathbb{N}^d\setminus S$ is the set of \emph{gaps} of $S$ and $\operatorname{g}(S)=|\operatorname{H}(S)|$ is called the \emph{genus} of $S$.
\item $\operatorname{PF}(S)=\{\mathbf{x}\in \operatorname{H}(S)\mid \mathbf{x}+\mathbf{s}\in S\  \mbox{for all}\ \mathbf{s}\in S\setminus \{\mathbf{0}\}$ is the set of \emph{pseudo-Frobenius} elements of $S$ and $\operatorname{t}(S)=|\operatorname{PF}(S)|$ is called the \emph{type} of $S$.
\item $\operatorname{SG}(S)=\{\mathbf{x}\in \operatorname{PF}(S)\mid 2\mathbf{x} \in S\}$ is the set of \emph{special gaps} of $S$.
\end{itemize}


It is known that a numerical semigroup $S$ is irreducible if and only if $\operatorname{SG}(S)=\{\operatorname{F}(S)\}$, and in such a case it can occur either $\operatorname{t}(S)=\{\operatorname{F}(S)\}$ or  $\operatorname{t}(S)=\{\operatorname{F}(S),\operatorname{F}(S)/2\}$. In the first case $S$ is called \emph{symmetric}, in the second case $S$ is called \emph{pseudo-symmetric} (see \cite[Chapter 3]{rosales2009numerical}). Irreducible GNSs are studied in \cite{cisto2019irreducible}.\\
We say that the set $A\subseteq \mathbb{N}^d$ generates $S$ if $S=\{\sum_{i=1}^e n_i\mathbf{a}_i\mid \mathbf{a}_i\in A, n_i \in \mathbb{N}\ \mbox{for all}\ i\in [e], e\in \mathbb{N}\}$, where as usual we denote $[e]=\{1,2,\ldots,e\}$ for $e\in \mathbb{N}$. If no proper subset of $A$ generates $S$ then $A$ is called a minimal system of generators of $S$. It has been proved that every GNS has a unique finite minimal system of generators (see \cite[Proposition 2.3]{Analele}).\\ 
In $\mathbb{N}^d$ we consider the natural partial order: $\mathbf{x}\leq \mathbf{y}$ if and only if $\mathbf{y}-\mathbf{x}\in\mathbb{N}^d$, $\mathbf{x},\mathbf{y}\in \mathbb{N}^d$. We recall that a total order $\prec$ in $\mathbb{N}^d$ is a \emph{monomial order} if:
\begin{enumerate}
\item $\mathbf{0}\prec \mathbf{u}$, for all $\mathbf{u}\in \mathbb{N}^d$.
\item If $\mathbf{u},\mathbf{v}\in\mathbb{N}$ and $\mathbf{u}\prec \mathbf{v}$, then $\mathbf{u}+\mathbf{w}\prec \mathbf{v}+\mathbf{w}$ for all $\mathbf{w}\in \mathbb{N}^d$.
\end{enumerate}

Different examples of monomial orders are provided in \cite{failla2016algorithms}. Consider now the following definitions: 

\begin{itemize}
\item Let $\operatorname{G}(S)$ be the set of minimal generators of $S$ and $\operatorname{e}(S)=|\operatorname{G}(S)|$, called \emph{embbeding dimension}.
\item Let $\operatorname{N}(S)=\{\mathbf{s}\in S\mid \mathbf{s}\leq \mathbf{h}\ \mbox{for some}\ \mathbf{h}\in \operatorname{H}(S)\}$ and denote $\operatorname{n}(S)=|\operatorname{N}(S)|$.
\item Let $\operatorname{c}(S)=|\{\mathbf{n}\in \mathbb{N}^d\mid \mathbf{n}\leq \mathbf{h}\ \mbox{for some}\ \mathbf{h}\in \operatorname{H}(S)\}$ 
\end{itemize}

The three invariants defined above are involved in the generalization for GNSs of a well known conjecture on numerical semigroups, called Wilf's conjecture, that states $\operatorname{e}(S)\operatorname{n}(S)\geq \operatorname{F}(S)+1$. This is still an open problem (see \cite{delgado2019conjecture} for a survey). Such a conjecture has been generalized for GNSs in \cite{cisto2020generalization} as:
\begin{conjecture}[Generalized Wilf's conjecture]
Let $S\subseteq \mathbb{N}^d$ be a GNS. Then $$\operatorname{e}(S)\operatorname{n}(S)\geq d \operatorname{c}(S) \ \ \ \mbox{or equivalently}\ \ \  (\operatorname{e}(S)-d)\operatorname{n}(S)\geq d\operatorname{g}(S)$$. 
\end{conjecture} 

We want to consider a sufficient condition for a GNS to satisfy the generalized Wilf's conjecture, introduced for the first time in the Ph.D thesis of the first author (see \cite{thesis}), that we report here with its proof for completeness. To obtain such a condition we need to consider the following property: 

\begin{proposition} Let $S\subseteq \mathbb{N}^{d}$ be a GNS and let $\operatorname{t}(S)=|\operatorname{PF}(S)|$. Then $\operatorname{g}(S)\leq \operatorname{t}(S)\operatorname{n}(S)$.
\end{proposition}
\begin{proof}
Consider in $\mathbb{N}^{d}$ a monomial order $\prec$. Let $\textbf{x}\in \operatorname{H}(S)$, we define $\textbf{f}_{\textbf{x}}=\min_{\prec}\{\textbf{f}\in \operatorname{PF}(S)\mid \textbf{x}\leq_{S}\textbf{f}\}$. The previous set is not empty by \cite[Proposition 1.3]{cisto2019irreducible}. So we can consider the function 
$$\phi:\operatorname{H}(S)\longrightarrow \operatorname{PF}(S)\times \operatorname{N}(S),\ \ \ \ \ \textbf{x}\longmapsto(\textbf{f}_{\textbf{x}},\textbf{f}_{\textbf{x}}-\textbf{x})$$
It is easy to see that $\phi$ is injective so $\operatorname{g}(S)\leq \operatorname{t}(S)\operatorname{n}(S)$. 
\end{proof}

Since $\operatorname{c}(S)=\operatorname{g}(S)+\operatorname{n}(S)$, then $\operatorname{c}(S)\leq (\operatorname{t}(S)+1)\operatorname{n}(S)$. So we can state the following:

\begin{corollary} Let $S\subseteq \mathbb{N}^{d}$ be a GNS and $\operatorname{t}(S)=|\operatorname{PF}(S)|$. If $\operatorname{e}(S)\geq d(\operatorname{t}(S)+1)$ then $S$ satisfies generalized Wilf's conjecture.
\label{wilfType}
\end{corollary}

Actually, the previous property is a generalization of a well known property of numerical semigroups (see \cite{dobbs2003question}) and we do not know till now any general class of GNSs satisfying it. In the next section we introduce a class of GNSs whose semigroups satisfy the condition of the above corollary. The following lemma will be useful.

\begin{lemma} Let $\alpha_{1}\geq\alpha_{2}\geq\cdots\geq \alpha_{d}$ be elements in $\mathbb{N}$ such that $\sum_{i=1}^{d}\alpha_{i}>g$, $g$ nonzero integer. Then there exist $\beta_{1},\beta_{2},\ldots,\beta_{d}$ such that $\alpha_{i}-\beta_{i}\in\mathbb{N}$ and $\sum_{i=1}^{d}\beta_{i}=g$. \label{numeri}
\end{lemma}

\begin{proof}

If there exists $j\in \{1,2,\ldots,d\}$ such that $\alpha_{j}\geq g$ then we can consider $\beta_{j}=g$ and $\beta_{i}=0$ for every $i\in \{1,2,\ldots,d\}\setminus \{j\}$. If $\alpha_{i}<g$ for all $i=1,2,\ldots,d$ then let $r^{(1)}\in \{1,\ldots,d\}$ such that $\alpha_{i}=0$ for all $i=r^{(1)}+1,r^{(1)}+2,\ldots,d$. If $r^{(1)}\geq g$ then we fix $\beta_{i}=1$ for $i=1,\ldots,g$ and $\beta_{j}=0$ for $j=g+1\ldots,d$. So $\alpha_{i}-\beta_{i}\geq 0$ for all $i\in \{1,\ldots,d\}$ and $\sum_{i=1}^{d}\beta_{i}=g$. If $r^{(1)}<g$ we consider the following steps:\\
\underline{First step}. Put $\gamma_{i}^{(1)}=1$ for $i \in\{1,\ldots,r^{(1)}\}$ and $\gamma_{j}^{(1)}=0$ for $j \in\{r^{(1)}+1,\ldots,d\}$. Let $\alpha_{i}^{(1)}=\alpha_{i}-\gamma_{i}^{(1)}\geq 0$ for $i=1,\ldots,d$ and $\triangle^{(1)}=g-r^{(1)}$. Observe that $$\sum_{i=1}^{d}\alpha_{i}^{(1)}=\sum_{i=1}^{d}\alpha_{i}-\sum_{i=1}^{d}\gamma_{i}^{(1)}>g-r^{(1)}>0,$$
 in particular there exists $j\in \{1,\ldots,d\}$ such that $\alpha_{j}^{(1)}\neq 0$.\\
\underline{Second step}. Let $r^{(2)}\in \{1,\ldots,r^{(1)}\}$ such that $\alpha_{i}^{(1)}=0$ for $i\in\{r^{(2)}+1,\ldots,d\}$. If $r^{(2)}\geq\triangle^{(1)}$ then we fix $\gamma_{i}^{(2)}=1$ for $i=1,\ldots,\triangle^{(1)}$ and $\gamma_{j}^{(2)}=0$ for $j=\triangle^{(1)}+1,\ldots,d$. We consider $\beta_{i}=\gamma_{i}^{(1)}+\gamma_{i}^{(2)}$ for $i=1,\ldots,d$ and we have $\alpha_{i}-\beta_{i}=\alpha_{i}^{(1)}-\gamma_{i}^{(2)}\geq 0$ for every $i\in \{1,\ldots,d\}$ and $\sum_{i}^{d}\beta_{i}=r^{(1)}+\triangle^{(1)}=g$.
If $r^{(2)}<\triangle^{(1)}$ we define $\gamma_{i}^{(2)}=1$ for $i=1,\ldots, r^{(2)}$, $\gamma_{j}^{(2)}=0$ for $j=r^{(2)}+1,\ldots,d$ and $\alpha_{i}^{(2)}=\alpha_{i}^{(1)}-\gamma_{i}^{(2)}\geq 0$ for $i=1,\ldots,d$. We put $\triangle^{(2)}=\triangle^{(1)}-r^{(2)}=g-r^{(1)}-r^{(2)}>0$ and observe that $\sum_{i}^{d}\alpha_{i}^{(2)}>g-r^{(1)}-r^{(2)}>0$, so there exists $j\in \{1,\ldots,d\}$ such that $\alpha_{j}^{(2)}\neq 0$. Therefore we can repeat the procedure from the beginning of the second step, considering the greatest index $r^{(3)}\in \{1,\ldots,r^{(2)}\}$ such that $\alpha_{i}^{(2)}=0$ for $i \in\{r^{(3)}+1,\ldots,d\}$ and considering the two cases $r^{(3)}\geq \triangle^{(2)}$ (and in this case we conclude) or $r^{(3)}< \triangle^{(2)}$, and so on. After a finite number $h$ of steps, it occurs that $r^{(h)} \geq \triangle^{(h-1)}$ (because it is impossible to obtain $g-r^{(1)}-\cdots-r^{(h)}>0$ for infinitely many steps) since $r^{(j)}>0$ for every $j$. Since $r^{(h)}\geq \triangle^{(h-1)}$, we obtain $\beta_{i}=\sum_{j=1}^{h}\gamma_{i}^{(j)}$ for every $i=1,\ldots,d$ and these elements satisfy the requested condition.

\end{proof}

\noindent The following example shows the procedure in the proof of the previous lemma.

\begin{example}
Let $d=4$, $g=10$ and consider $\alpha_{1}=8,\alpha_{2}=7,\alpha_{3}=3,\alpha_{4}=2$. We have $\sum_{i=1}^{4}\alpha_{i}=20>g$. Moreover $\alpha_{i}<g$ for $i=1,2,3,4$.\\
We have $r^{(1)}=4<g$. So we define $\gamma_{i}^{(1)}=1$ for $i=1,2,3,4$ and consider the following positive integers:
\begin{itemize}
\item $\alpha_{1}^{(1)}=\alpha_{1}-\gamma_{1}^{(1)}=7$
\item $\alpha_{2}^{(1)}=\alpha_{2}-\gamma_{2}^{(1)}=6$
\item $\alpha_{3}^{(1)}=\alpha_{3}-\gamma_{3}^{(1)}=2$
\item $\alpha_{4}^{(1)}=\alpha_{4}-\gamma_{4}^{(1)}=1$.
\end{itemize}
We have $\triangle^{(1)}=g-r^{(1)}=6$ and put $r^{(2)}=4<\triangle^{(1)}$. So in the second step we consider $\gamma_{i}^{(2)}=1$ for $i=1,2,3,4$ and the following:
\begin{itemize}
\item $\alpha_{1}^{(2)}=\alpha_{1}^{(1)}-\gamma_{1}^{(2)}=6$
\item $\alpha_{2}^{(2)}=\alpha_{2}^{(1)}-\gamma_{2}^{(2)}=5$
\item $\alpha_{3}^{(2)}=\alpha_{3}^{(1)}-\gamma_{3}^{(2)}=1$
\item $\alpha_{4}^{(2)}=\alpha_{4}^{(1)}-\gamma_{4}^{(2)}=0$.
\end{itemize}
We have $\triangle^{(2)}=\triangle^{(1)}-r^{(2)}=g-r^{(1)}-r^{(2)}=2$ and define $r^{(3)}=3>\triangle^{(2)}$. So the next step is the last, in which $\gamma_{1}^{(3)}=1$, $\gamma_{2}^{(3)}=1$, $\gamma_{3}^{(3)}=0$, $\gamma_{4}^{(3)}=0$. We conclude defining:
\begin{itemize}
\item $\beta_{1}=\gamma_{1}^{(1)}+\gamma_{1}^{(2)}+\gamma_{1}^{(3)}=3$
\item $\beta_{2}=\gamma_{2}^{(1)}+\gamma_{2}^{(2)}+\gamma_{2}^{(3)}=3$
\item $\beta_{3}=\gamma_{3}^{(1)}+\gamma_{3}^{(2)}+\gamma_{3}^{(3)}=2$
\item $\beta_{4}=\gamma_{4}^{(1)}+\gamma_{4}^{(2)}+\gamma_{4}^{(3)}=2$.
\end{itemize}

\end{example}

In the following we denote by $\mathbf{e}_1,\mathbf{e}_2,\ldots,\mathbf{e}_d$ the standard basis vectors of the vector space $\mathbb{R}^d$.\\
Let $\mathbf{x}=(x^{(1)},x^{(2)},\ldots,x^{(d)})\in \mathbb{N}^{d}$. Simplifying our notation, we define along the paper $|\mathbf{x}|:=||\mathbf{x}||_1$, that is $|\mathbf{x}|=\sum_{i=1}^{d}x^{(i)}$.
Finally, recall that if $A=\{\mathbf{x}\in \mathbb{N}^d \mid |\mathbf{x}|=i\}$, it is well known that $|A|=\binom{i+d-1}{d-1}$. 

\section{$T$-stripe GNSs}

\begin{definition}\label{Definition: T,m Stripe}
Let $T$ be a numerical semigroup. Let $\operatorname{H}_0=\{\mathbf{x}\in \mathbb{N}^d:0<|\fx|<\operatorname{m}(T)\}$ and $\operatorname{H}_i=\{h\mathbf{e}_i:h\in \operatorname{H}(T)\}$ for $i\in[d]$. We set $H=\cup_{i=0}^{d}\operatorname{H}_i$. It is trivial to see that $S=\mathbb{N}^d\backslash H$ is a generalized numerical semigroup that we call \textit{$T$-stripe} GNS.
\end{definition}

If $S$ is a $T$-stripe GNS then we denote $\operatorname{K}(S)=\operatorname{H}_0\backslash\cup_{i=1}^d\operatorname{H}_i$.

\begin{proposition}\label{Prop: Type T,m Stripe}
	Let $S$ be the $T$-stripe GNS. Then 
	$$ \operatorname{PF}(S)=\bigcup_{i=1}^d\{f\mathbf{e}_i:f\in \operatorname{PF}(T)\}\cup\operatorname{K}(S).$$
	In particular $\operatorname{t}(S)=|\operatorname{H}_0|-d(m-1)+d\cdot \operatorname{t}(T)$.
\end{proposition}

\begin{proof}
	We set $\mathcal{L}=\bigcup_{i=1}^d\{f\mathbf{e}_i:f\in \operatorname{PF}(T)\}\cup\operatorname{K}(S)$. We prove that $\operatorname{PF}(S)\subseteq \mathcal{L}$. Consider $\fx\in \operatorname{PF}(S)$. Then $\fx\in \operatorname{H}(S)$ and $\fx+\mathbf{s}\in S^{*}$ for all $\mathbf{s}\in S^*$. Since $\fx \in \operatorname{H}(S)$, we have $\fx\in \operatorname{H}_0$ or $\fx\in \operatorname{H}_i$ for some $i\in[d]$. Suppose that $\fx\in \operatorname{H}_i$, for some $i\in [d]$. Then $\fx=h\mathbf{e}_i$ for some $h\in \operatorname{H}(T)$. Set $s\in T^*$, then $s\mathbf{e}_i\in S^*$ and $\mathbf{x}+s\mathbf{e}_i=(h+s)\mathbf{e}_i\in S^*$. So it is easy to obtain that $h+s\in T^*$ for all $s\in T^*$, that is $h\in \operatorname{PF}(T)$. In particular $\mathbf{x}\in \{f\mathbf{e}_i:f\in \operatorname{PF}(T)\}$. Assume that $\fx\notin \operatorname{H}_i$ for all $i\in [d]$, so $\fx \in \operatorname{H}_0$. Then trivially $\fx \in \operatorname{K}(S)$. Now we prove that $\mathcal{L}\subseteq \operatorname{PF}(S)$. Let $\fx \in \mathcal{L}$ and $\mathbf{s}\in S^*$. Observe that $\fx+\mathbf{s}\notin \operatorname{H}_0$, because $|\fx+\mathbf{s}|=|\fx|+|\mathbf{s}|\geq |\fx|+m\geq m$. In particular if $\mathbf{x}\in \operatorname{K}(S)$ we have also $\mathbf{x}+\mathbf{s}\notin \operatorname{H}_i$ for all $i\in [d]$, that is $\mathbf{x}+\mathbf{s}\in S$. Assume that $\fx=f\mathbf{e}_i$ for some $i\in [d]$ and $f\in \operatorname{PF}(T)$. Suppose that $\fx+\mathbf{s}\in \operatorname{H}(S)\backslash \operatorname{H}_0$. Then $\fx+\mathbf{s}\in \operatorname{H}_i\setminus \operatorname{H}_0$ for some $i$, in particular $\mathbf{s}=s\mathbf{e}_i$ with $s\in T^*$, that leads to a contradiction since we obtain $f+s\in H(T)$ but $f\in \operatorname{PF}(T)$. Hence $\fx +\mathbf{s}\in S^*$, so $\fx \in \operatorname{PF}(S)$. The last statement on $\operatorname{t}(S)$ easily follows. 
\end{proof}

In the following, for $i\in \mathbb{N}$, as usual we denote $\mathbb{N}\mathbf{e}_i=\{x\mathbf{e}_i \mid x\in\mathbb{N}\}$.

\begin{proposition}\label{Prop: generatori T,m Stripe}
	Let $S$ be the $T$-stripe GNS in $\mathbb{N}^d$. Let $A_{0}=\{\fx\in \mathbb{N}^d\mid m\leq |\fx|\leq 2m-1,\fx\notin \mathbb{N}\mathbf{e}_i\ \forall i \in [d]\}$ and $A_i=\{t\mathbf{e}_i\mid t\in \operatorname{G}(T)\}$ for all $i\in [d]$. Then $\operatorname{G}(S)=\bigcup_{i=0}^{d}A_i$, and in particular $$\operatorname{e}(S)=\sum_{i=m}^{2m-1}\binom{i+d-1}{d-1}-d[m-\operatorname{e}(T)].$$
\end{proposition}
\begin{proof}
	Firstly we prove that $\bigcup_{i=0}^{d}A_i\subseteq \operatorname{G}(S)$. Let $\fx\in \bigcup_{i=0}^{d}A_i$. We may assume that $\fx\in A_0$, because the other case is trivial. If we suppose that there exist $\mathbf{y},\mathbf{z}\in S^*$ such that $\fx=\mathbf{y}+\mathbf{z}$, then $|\fx|=|\mathbf{y}+\mathbf{z}|\geq 2m$, a contradiction since $|\fx|\leq2m-1$. Hence $\fx\in \operatorname{G}(S)$. Now we prove that $\operatorname{G}(S)\subseteq\bigcup_{i=0}^{d}A_i$. Let $\fx=(x^{(1)},\dots,x^{(d)})\in \operatorname{G}(S)$. We suppose that $\fx \notin \bigcup_{i=0}^{d}A_i$. Obviously $\fx$ cannot be on the $i$-th axis of $\mathbb{N}^d$ because it is generated by $A_i$ and in such a case $\fx$ cannot be a minimal generator of $S$. Then $\fx \notin \cup_{i=1}^d\mathbb{N}\mathbf{e}_i$ and $|\fx|\geq 2m$. From Lemma \ref{numeri} it follows that there exists $\mathbf{b}\in \mathbb{N}^d$ with $|\mathbf{b}|=m$ such that $\fx-\mathbf{b}\in \mathbb{N}^d$, so $\fx=\mathbf{a}+\mathbf{b}$ where $\mathbf{a}\in \mathbb{N}^d$ and $|\mathbf{a}|=|\fx-\mathbf{b}|\geq m$. Observe that $\mathbf{b}\in S$ since $|\mathbf{b}|=m$. If $\mathbf{a}\in S$ then $\mathbf{x}\notin \operatorname{G}(S)$, a contradiction. Then $\mathbf{a}\notin S$, in particular $\mathbf{a}\in \cup_{i=1}^{d}\operatorname{H}_i$ because $|\mathbf{a}|\geq m$. Assume that $\mathbf{a}\in \operatorname{H}_i$ for some $i\in[d]$, so $\mathbf{a}=a_i\mathbf{e}_i$ with $a_i\geq m$. As a consequence, since $\mathbf{x}=\mathbf{a}+\mathbf{b}$, we have $x^{(i)}\geq m$. In such a case let $\mathbf{w}=\mathbf{x}-m\mathbf{e}_i$. Observe that $|\mathbf{w}|=|\mathbf{x}-m\mathbf{e}_i|\geq m$. We distinguish two cases. In the first, suppose that $\fx$ has more than two non-null components. Then $\mathbf{w}$ has at least two non-null components , since $\fx=\mathbf{w}+m\mathbf{e}_i$, and therefore $\mathbf{w}\in S$. Then $\fx\notin \operatorname{G}(S)$, a contradiction. In the second case suppose that $\mathbf{x}$ has only two non-null components, so $\mathbf{w}=x^{(k)}\mathbf{e}_k+(x^{(i)}-m)\mathbf{e}_i$ for some $k\in [d]\setminus \{i\}$. If $x^{(i)}>m$ then $\mathbf{w}\in S^*$, hence $x\notin \operatorname{G}(S)$ since $\mathbf{x}=\mathbf{w}+m\mathbf{e}_i$, a contradiction. If $x^{(i)}=m$, then $x^{(k)}\geq m$ because $|\fx|\geq 2m$. Let $r=\min \{j\in \mathbb{N}:jm>x^{(k)}\}$. Then $\fx=\mathbf{u}+\mathbf{v}$, where $\mathbf{u}=(x^{(k)}-m)\mathbf{e}_k+(rm-x^{(k)})\mathbf{e}_i)\in S^*$ and $\mathbf{v}=m\mathbf{e}_k+(x^{(k)}-(r-1)m)\mathbf{e}_i\in S^*$. Hence $\fx\notin \operatorname{G}(S)$, a contradiction again. All cases lead to a contradiction, so necessarily $\fx\in \cup_{i=0}^d A_i$. The embedding dimension can be easily computed.    
\end{proof}

\begin{example}
	Let $T$ be the numerical semigroup generated by $5,6$ and $13$. Observe that $T=\mathbb{N}\backslash\{1,2,3,4,7,8,9,13,14\}$. The $T$-stripe GNS in $\mathbb{N}^2$ is generated by $ ( 5, 0 ), ( 6, 0 ), \newline ( 13, 0 ), (0, 5), ( 0, 6 ),( 0, 13 ), ( 4, 1 ), ( 3, 2 ), ( 2, 3 ), ( 1, 4 ), ( 5, 1 ), ( 4, 2 ), ( 3, 3 ), ( 2, 4 ), ( 1, 5 ), ( 6, 1 ), \newline ( 5, 2 ),  ( 4, 3 ), ( 3, 4 ), ( 2, 5 ), ( 1, 6 ), ( 7, 1 ), ( 6, 2 ), ( 5, 3 ), ( 4, 4 ), ( 3, 5 ), ( 2, 6 ), ( 1, 7 ),  ( 8, 1 ), ( 7, 2 ), \newline ( 6, 3 ), ( 5, 4 ),  ( 4, 5 ), ( 3, 6 ), ( 2, 7 ), ( 1, 8 )$.\\ In Figure \ref{Figura: esempio T,m-stripe} we provide a graphical view of $S$: the red points are the holes of $S$, the grey ones are the minimal generators. The blue points represent other elements of $S$.
	\begin{figure}[h]
		\centering
		\includegraphics[scale=0.75]{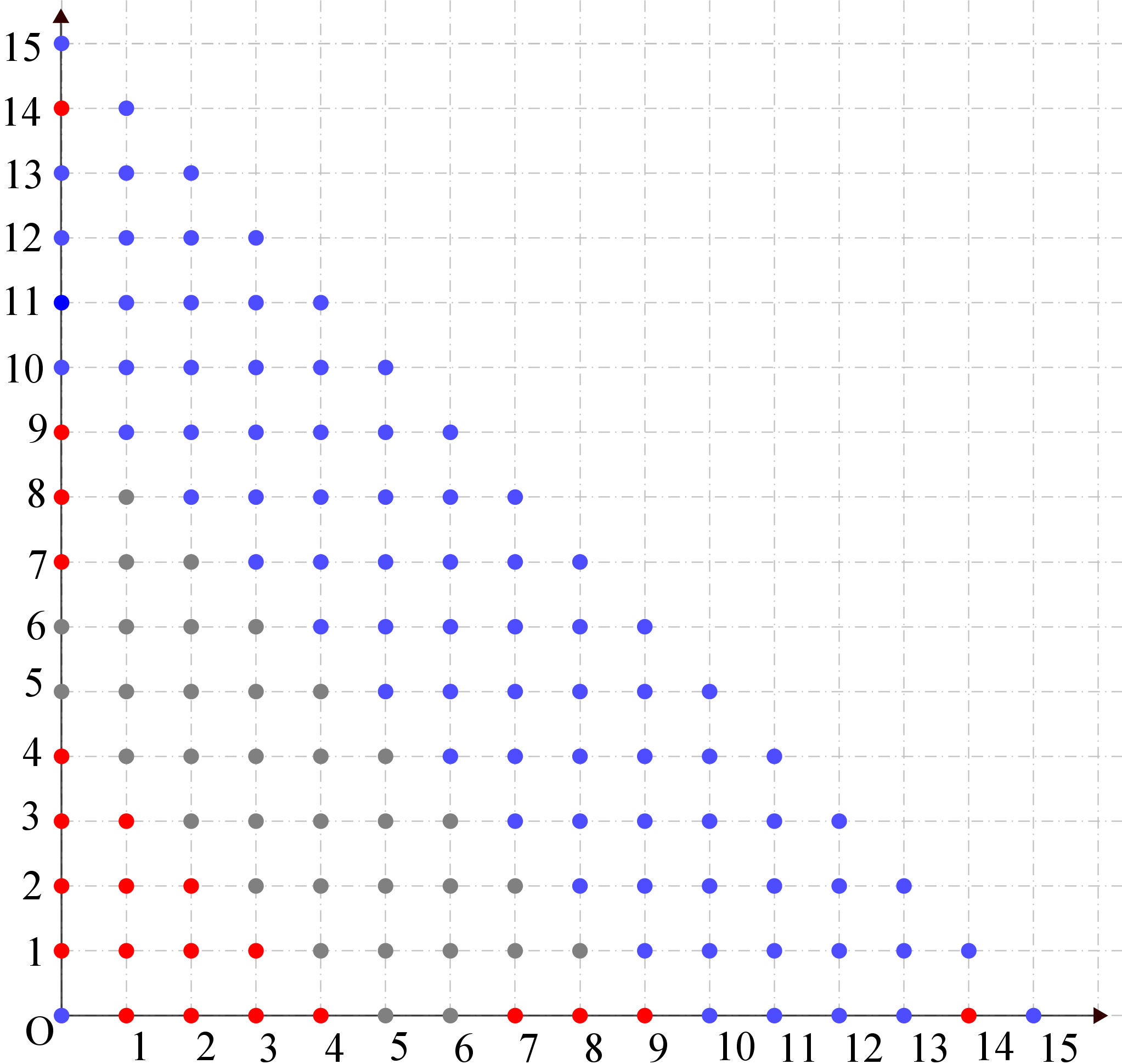}
		\caption{}
		\label{Figura: esempio T,m-stripe}
	\end{figure}
\end{example}

Now we want to study the generalized Wilf's conjecture for $T$-stripe GNSs, proving that for all numerical semigroups $T$ and for all $d\geq 2$ the $T$-stripe GNS $S$ satisfies the inequality $\operatorname{e}(S)\geq d(\operatorname{t}(S)+1)$.\\ 
First of all set $B_d(m)=\binom{m+d-1}{d-1}$, in order to simplify our notations. Applying repeatedly the known equality $\binom{n}{k}=\binom{n}{k-1}+\binom{n-1}{k-1}$, it is easy to obtain that $B_d(m)=\sum_{i=0}^{m}B_{d-1}(i)$.\\
Let $T$ be a numerical semigroup having multiplicity $m$ and $S$ be the $T$-stripe GNS in $\mathbb{N}^d$, $d\geq 2$. We denote by $e$ and $t$ respectively the embedding dimension and the type of $T$. From the previous considerations and from Proposition \ref{Prop: Type T,m Stripe} and \ref{Prop: generatori T,m Stripe} it follows that the inequality $\operatorname{e}(S)\geq d(\operatorname{t}(S)+1)$ is equivalent to the following 
$$ B_{d+1}(2m-1)-(d+1)B_{d+1}(m-1)\geq \Theta_d$$
where $\Theta_d:=d[m-e-d(m-1-t)]$. \\
Our aim is to prove the previous inequality for every choice of the numerical semigroup $T$ and for all $d\in \mathbb{N}$ with $d\geq 2$.

\begin{remark} \rm
Consider the numerical semigroup $T=\mathbb{N}\setminus \{1,\ldots,m-1\}$. It is well known that $\operatorname{e}(T)=m$. The $T$-stripe GNS $S$ in $\mathbb{N}^d$ has $n(S)=1$, so by \cite[Theorem 5.7]{cisto2020generalization} it follows that $S$ satisfies the generalized Wilf's conjecture, hence that $\operatorname{e}(S)\geq d\operatorname{c}(S)$. In particular $\sum_{i=m}^{2m-1}\binom{i+d-1}{d-1}\geq d\sum_{j=0}^{m-1}\binom{j+d-1}{d-1}$, equivalently 
$B_{d+1}(2m-1)-(d+1)B_{d+1}(m-1)\geq 0$ for $m\geq 2$ and $d\geq 2$. This means that the left hand side of the inequality, which we want to prove, is always positive. 
\label{Remark-left}
\end{remark}

%

	We recall that for a numerical semigroup $T$ we have $\operatorname{e}(T)\leq \operatorname{m}(T)$ and $\operatorname{t}(T)\leq \operatorname{m}(T)-1$. Moreover $\operatorname{e}(T)=\operatorname{m}(T)$ if and only if $\operatorname{t}(T)=\operatorname{m}(T)-1$ and in such a case $T$ is said to have \emph{maximal embedding dimension} (see \cite[Corollary 3.2]{rosales2009numerical}).
	
\begin{lemma}\label{Lemma: Condizioni sufficiente A_d<0}
	Let $T$ be a numerical semigroup. Suppose that one of the following conditions holds:
	\begin{enumerate}
	\item $\operatorname{e}(T)\in\{2,3\}$;
	\item $T$ is a numerical semigroup with maximal embedding dimension; 
	\item $\operatorname{m}(T)=2$.
	\end{enumerate}
 	Then the $T$-stripe GNS $S$ satisfies $\operatorname{e}(S)\geq d(\operatorname{t}(S)+1)$.
\end{lemma}
\begin{proof}
	$(1)$ Assume that $\operatorname{e}(T)=2$. Then $T$ is symmetric and by \cite[Corollary 10.22]{rosales2009numerical} $\operatorname{t}(T)=1$, so $\Theta_d=d(1-d)(m-2)\leq 0$. Assume $\operatorname{e}(T)=3$, then $\operatorname{t}(T)\in \{1,2\}$ by \cite[Corollary 10.22]{rosales2009numerical}. If $\operatorname{t}(T)=1$, then $\Theta_d=d[m-3-d(m-2)]\leq0$. If $\operatorname{t}(T)=2$, then $\Theta_d=d[m-3-d(m-3)]\leq0$. \\
	$(2)$ If $T$ is a numerical semigroup with maximal embedding dimension then $\operatorname{t}(T)=\operatorname{m}(T)-1$ and $\operatorname{e}(T)= \operatorname{m}(T)$, so $\Theta_d=0$.\\
	$(3)$ In this case $\operatorname{e}(T)=2$, and the assertion follows from (1).\\
	In all the previous cases we have $\operatorname{e}(S)\geq d(\operatorname{t}(S)+1)$ by Remark~\ref{Remark-left}.
\end{proof}

Now we need a refinement of the inequality in Remark~\ref{Remark-left}. 

\begin{lemma}\rm \label{Lemma: Disequaglianza stripe con l'ordinario}
    Let $m\geq 3$ and $d\geq 2$ be two integers. Then $$B_{d+1}(2m-1)-(d+1)B_{d+1}(m-1)\geq d(d-1).$$
\end{lemma}
\begin{proof} 
	 
	Let $T=\mathbb{N}\setminus \{1,2,\ldots, m-1\}$ and $S$ the $T$-stripe GNS in $\mathbb{N}^d$ as in Remark~\ref{Remark-left}. We prove that $\operatorname{e}(S)-d\operatorname{c}(S)\geq d(d-1)$, then the conclusion follows arguing as in Remark~\ref{Remark-left}. Let $M=\{m\mathbf{e}_i+\mathbf{h}\mid \mathbf{h}\in \operatorname{H}(S)\cup \{\mathbf{0}\}, i\in [d]\}$ and observe that $|M|=d \operatorname{c}(S)$. Moreover $M\subseteq \operatorname{G}(S)$, so $\operatorname{e}(S)-d\operatorname{c}(S)=|\operatorname{G}(S)\setminus M|$. Consider the set $A=\{(m-1)\mathbf{e}_i+\mathbf{e}_j\mid i,j\in [d], i\neq j\}$ and observe that $A\subseteq \operatorname{G}(S)$ and $A\cap M=\emptyset$. Furthermore $|A|=d(d-1)$, so we obtain our claim. 
\end{proof} 


\begin{remark}
From the proof of the previous lemma it follows that if $S$ is the $T$-stripe GNS in $\mathbb{N}^d$ with $T=\mathbb{N}\setminus \{1,2,\ldots, m-1\}$, $m\geq 2$, then we can estimate a lower bound for the difference $\operatorname{e}(S)-d\operatorname{c}(S)$. Such a proof is inspired by some arguments contained in the proof of \cite[Proposition 8.1]{cisto2020generalization}. In the case $d=2$ indeed we obtain a class of GNSs for which we can apply directly \cite[Proposition 8.1]{cisto2020generalization}, obtaining the exact value of $\operatorname{e}(S)-2\operatorname{c}(S)$. In the framework of \cite[Section 5]{cisto2020generalization}, since $\operatorname{n}(S)=1$, then $S$ is a monomial semigroup with correspondent ideal $I=(x_1^{n_1}\cdots x_d^{n_d}\mid (n_1,\ldots,n_d)\in S\setminus \{\mathbf{0}\})$ in $R=K[x_1,\ldots,x_d]$, $K$ a field. Consider also the ideal $J=(x_1^{a_1},\ldots,x_d^{a_d})$ where $a_1,\ldots,a_d$ are the smallest integers such that $x_i^{a_i}\in I$ for each $i\in[d]$. It is easy to see the $a_i=m$ for all $i\in[d]$. To apply \cite[Proposition 8.1]{cisto2020generalization} we have to verify that $I^2=IJ$. Observe that $IJ\subseteq I^2$. The converse is true if and only if for all $\bm{\alpha}=(\alpha_{1},\ldots,\alpha_{d}),\bm{\beta}=(\beta_{1},\ldots,\beta_{d})\in \mathbb{N}^d$ such that $|\bm{\alpha}|\geq m$ and $|\bm{\beta}|\geq m$ the equality
$\bm{\alpha}+\bm{\beta}=\bm{\gamma}+m\mathbf{e}_i$ holds for some $i\in[d]$ and $\bm{\gamma}\in \mathbb{N}^d$, equivalently there exists $i\in [d]$ such that $\alpha_i+\beta_i\geq m$. It is not difficult to see that this is true if $d=2$, not for $d>2$. In fact if $d=3$, $m=7$, $\bm{\alpha}=5\mathbf{e}_1+2\mathbf{e}_3$ and $\bm{\beta}=5\mathbf{e}_2+2\mathbf{e}_3$, it is not $\alpha_i+\beta_i\geq 7$ for any $i\in [3]$. So for $T=\mathbb{N}\setminus \{1,2,\ldots, m-1\}$, $m\in \mathbb{N}$, $S$ the $T$-stripe in $\mathbb{N}^2$, we obtain $\operatorname{e}(S)-d\operatorname{c}(S)=m^2-\operatorname{c}(S)$.
\end{remark}

\begin{lemma}\label{Lemma: la disuguaglianza  maggiore di d(m-6)}
	Let $m\geq 7$, $d\geq 2$ be two integers. Then 
	$$B_{d+1}(2m-1)-(d+1)B_{d+1}(m-1)\geq d(m-6).$$
\end{lemma}

\begin{proof}
	Fix $d\geq 2$. Define the sequence $a_m=B_{d+1}(2m-1)-(d+1)B_{d+1}(m-1)-d(m-6)$ for all $m\geq 7$. In order to obtain our claim it suffices to prove that the sequence $\{a_m\}_{m\geq7}$ is increasing and $a_7\geq 0$.\\
	Firstly we show that $\{a_m\}_{m\geq7}$ is increasing. For all $m\geq 7$ we have to prove that $a_{m+1}\geq a_m$, that is
 $$	B_{d+1}(2m+1)-(d+1)B_{d+1}(m)-d(m-5)\geq B_{d+1}(2m-1)-(d+1)B_{d+1}(m-1)-d(m-6)$$
	Since $B_d(m)=\sum_{i=0}^{m}B_{d-1}(i)$, we obtain that $B_{d+1}(2m+1)-B_{d+1}(2m-1)=B_{d}(2m+1)+B_d(2m)$ and $B_{d+1}(m)-B_{d+1}(m-1)=B_{d}(m)$, so we have the following equivalent inequality $$B_{d}(2m+1)+B_d(2m)-(d+1)B_d(m)-d\geq 0$$ 
	It is trivially true for $d=2$. For $d=3$ it is equivalent to $5m^2+m-8\geq 0$ that is also true for $m\geq 7$. For $d\geq 4$ we can rewrite the left hand side of the previous inequality as $B_{d}(2m+1)-dB_d(m)-d+B_d(2m)-B_d(m)$. By Lemma~\ref{Lemma: Disequaglianza stripe con l'ordinario}, we have $B_{d}(2m+1)-dB_d(m)\geq (d-1)(d-2)$, obtained considering $d-1$ instead of $d$ and $m+1$ instead of $m$. Moreover, since $d\geq 4$, we have also $(d-1)(d-2)\geq d$ so $B_{d}(2m+1)-dB_d(m)-d\geq 0$. Moreover for $m\geq 7$ and $d\geq 4$ then $B_d(2m)-B_d(m)\geq 0$. Hence we have the first desired claim. It remains to prove that $a_7\geq 0$, that is $B_{d+1}(13)-(d+1)B_{d+1}(6)\geq d$, but this follows from Lemma~\ref{Lemma: Disequaglianza stripe con l'ordinario} for $m=7$. 
\end{proof}


Now we can state the following general result.

\begin{theorem}
	Let $T$ be a numerical semigroup and $S$ be the $T$-stripe GNS in $\mathbb{N}^d$, $d\geq 2$. Then $\operatorname{e}(S)\geq d(\operatorname{t}(S)+1)$. In particular $S$ satisfies generalized Wilf's conjecture.
\end{theorem}
\begin{proof}
By Lemma~\ref{Lemma: Condizioni sufficiente A_d<0} we can consider $\operatorname{e}(T)\geq 4$ and $\operatorname{t}(T)\leq \operatorname{m}(T)-2$. Under these assumptions, we have $\Theta_d\leq d[m-4-d(m-1)+dt]\leq d[m-4-d(m-1)+d(m-2)]\leq d(m-4-d)\leq d(m-6)$. So, in:
$$B_{d+1}(2m-1)-(d+1)B_{d+1}(m-1)\geq d(m-6)\geq \Theta_d.$$ the second inequality is true by our assumption and the left hand side of the first inequality is always positive by Remark~\ref{Remark-left}. So the first inequality is trivially true for $m\leq 6$, for $m\geq 7$ by Lemma~\ref{Lemma: la disuguaglianza  maggiore di d(m-6)}. So it is always $B_{d+1}(2m-1)-(d+1)B_{d+1}(m-1)\geq \Theta_d$ and this concludes our proof.
\end{proof}

\section{T-graded GNSs}
We introduce another class of generalized numerical semigroups related to a numerical semigroup $T$. In this section we denote by $G_i$ the set $G_i=\{\mathbf{x}\in \mathbb{N}^d \mid |\mathbf{x}|=i\}$. It is well known that $|G_i|=\binom{i+d-1}{d-1}$.\\
Let $T$ be a numerical semigroup. We put $S=\{\mathbf{a}\in \mathbb{N}^d \mid |\mathbf{a}|\in T\}$. It is easy to prove that $S$ is a generalized numerical semigroup in $\mathbb{N}^{d}$.

\begin{definition} \rm  Let $T$ be a numerical semigroup. We call $S=\{\mathbf{a}\in \mathbb{N}^d \mid |\mathbf{a}|\in T\}$ a $T$-graded GNS.
\end{definition}

\begin{remark} 
If $S\subseteq \mathbb{N}^{d}$ is a $T$-graded GNS then 
$$\operatorname{H}(S)=\{\mathbf{a}\in \mathbb{N}^d \mid |\mathbf{a}|\notin T\}=\bigcup_{i\in \operatorname{H}(T)}G_i\ \ \ \mbox{and}\ \ \ \operatorname{g}(S)=\sum_{i\in \operatorname{H}(T)}\binom{i+d-1}{d-1}$$
 Moreover if $\mathbf{x}\in \mathbb{N}^{d}$ then $\mathbf{x}\in S$ if and only if $|\mathbf{x}|\in T$.
\end{remark}

\begin{theorem} Let $T$ be a numerical semigroup such that $\operatorname{G}(T)=\{n_{1},\ldots,n_{r}\}$, and let $S\subseteq \mathbb{N}^{d}$ be the $T$-graded GNS. 
Then $\operatorname{G}(S)=\bigcup_{i=1}^{r}G_{n_{i}}$.
\label{teorstripe}
\end{theorem} 

\begin{proof}
Denote $G=\bigcup_{i=1}^{r}G_{n_{i}}$. We prove that any $\mathbf{a}\in S$ can be written as sum of elements in $G$. In particular $|\mathbf{a}|=\sum_{j=1}^{r}\lambda_{j}n_{j}$, with $\lambda_{j}\in \mathbb{N}$ for every $j=1,\ldots,r$. If $|\mathbf{a}|=n_{k}$ for some $k\in \{1,\ldots,r\}$ there is nothing to prove. Otherwise there exists $k\in \{1,\ldots,r\}$ such that  $|\mathbf{a}|>n_{k}$ and $\lambda_{k}>0$. In such a case, by Lemma~\ref{numeri}, there exists $\mathbf{b}\in \mathbb{N}^{d}$ such that $|\mathbf{b}|=n_{k}$ and $\mathbf{a}-\mathbf{b}\in \mathbb{N}^{d}$. So there exists $\mathbf{c}\in \mathbb{N}^{d}$ such that $\mathbf{a}=\mathbf{b}+\mathbf{c}$, moreover $|\mathbf{c}|=\sum_{j\neq k}\lambda_{j}n_{j}+(\lambda_{k}-1)n_{k}$, in particular $\mathbf{c}\in S$. Now, if $|\mathbf{c}|=n_{h}$ for some $h\in \{1,\ldots,r\}$ then $\mathbf{c}\in G$, otherwise there exists $h\in \{1,\ldots,r\}$ such that $|\mathbf{c}|>n_{h}$ and we can apply the same argument to $\mathbf{c}$. The computation stops after a finite number of steps because $|\mathbf{c}|< |\mathbf{a}|$. So every element of $S$ can be expressed as a sum of elements in $G$, that is $G$ is a set of generators for $S$. \\
Finally, if $\mathbf{a},\mathbf{b}\in G$ then the sum of coordinates of $\mathbf{a}+\mathbf{b}$ is sum of at least two elements in $T$, hence $|\mathbf{a}+\mathbf{b}|$ is not a minimal generator of $T$, that is $\mathbf{a}+\mathbf{b}\notin G$. So the set of generators $G$ is minimal for $S$.
  
\end{proof}

\begin{example} 

Let $T=\langle 4,6,7\rangle=\mathbb{N}\setminus \{1,2,3,5,9\}$ and let $S$ be the $T$-graded generalized numerical semigroup in $\mathbb{N}^{2}$. Then $S$ is generated by the set $G=\{(4,0),(3,1),(2,2),(1,3),\newline (0,4),(6,0),(5,1),(4,2),(3,3),(2,4),(1,5),(0,6),(7,0),(6,1),(5,2),(4,3),(3,4),(2,5),\newline(1,6),(0,7)\}$, and $\operatorname{H}(S)=\{(1,0),(0,1),(2,0),(1,1),(0,2),(3,0),(2,1),(1,2),(0,3),(5,0),\newline(4,1),(3,2),(2,3),(1,4),(0,5),(9,0),(8,1),(7,2),(6,3),(5,4),(4,5),(3,6),(2,7),(1,8),(0,9)\}$.

\vspace{12pt}

\begin{figure}[h!]
\begin{center}
\begin{tikzpicture}[scale=0.8] 
\usetikzlibrary{patterns}
\draw [help lines] (0,0) grid (10,10);
\draw [<->] (0,11) node [left] {$y$} -- (0,0)
-- (11,0) node [below] {$x$};
\foreach \i in {1,...,10}
\draw (\i,1mm) -- (\i,-1mm) node [below] {$\i$} 
(1mm,\i) -- (-1mm,\i) node [left] {$\i$}; 
\node [below left] at (0,0) {$O$};

\draw [mark=*] plot (0,1);
\draw [mark=*] plot (1,0);
\draw [mark=*] plot (0,2); 
\draw [mark=*] plot (1,1);
\draw [mark=*] plot (2,0);
\draw [mark=*] plot (3,0);
\draw [mark=*] plot (2,1);
\draw [mark=*] plot (1,2);
\draw [mark=*] plot (0,3);
\draw [mark=*] plot (5,0);
\draw [mark=*] plot (4,1);
\draw [mark=*] plot (3,2);
\draw [mark=*] plot (2,3);
\draw [mark=*] plot (1,4);
\draw [mark=*] plot (0,5);
\draw [mark=*] plot (9,0);
\draw [mark=*] plot (8,1);
\draw [mark=*] plot (7,2);
\draw [mark=*] plot (6,3);
\draw [mark=*] plot (5,4);
\draw [mark=*] plot (4,5);
\draw [mark=*] plot (3,6);
\draw [mark=*] plot (2,7);
\draw [mark=*] plot (1,8);
\draw [mark=*] plot (0,9);

\draw [red, mark=*] plot (4,0);
\draw [red, mark=*] plot (3,1);
\draw [red, mark=*] plot (2,2);
\draw [red, mark=*] plot (1,3);
\draw [red, mark=*] plot (0,4);
\draw [red, mark=*] plot (0,6);
\draw [red, mark=*] plot (1,5);
\draw [red, mark=*] plot (2,4);
\draw [red, mark=*] plot (3,3);
\draw [red, mark=*] plot (4,2);
\draw [red, mark=*] plot (5,1);
\draw [red, mark=*] plot (6,0);
\draw [red, mark=*] plot (7,0);
\draw [red, mark=*] plot (6,1);
\draw [red, mark=*] plot (5,2);
\draw [red, mark=*] plot (4,3);
\draw [red, mark=*] plot (3,4);
\draw [red, mark=*] plot (2,5);
\draw [red, mark=*] plot (1,6);
\draw [red, mark=*] plot (0,7);

\end{tikzpicture}
\end{center}
\caption{}
\label{figure-stripe}
\end{figure}

\noindent Figure~\ref{figure-stripe} provides a graphical view of S: black points are the holes of S, while the red points are the minimal generators. The other points are all elements in $S$.

\end{example}

\noindent Pseudo-Frobenius elements and special gaps are described in the following:   

\begin{proposition}
Let $T$ be a numerical semigroup and $S\subseteq \mathbb{N}^{d}$ be the $T$-graded GNS. Then:
\begin{enumerate}
\item $\operatorname{PF}(S)=\bigcup_{i\in \operatorname{PF}(T)}G_{i}$.
\item $\operatorname{SG}(S)=\bigcup_{i\in \operatorname{SG}(T)}G_{i}$.
\end{enumerate}
\label{pseudoFrobStripe}
\end{proposition}

\begin{proof}
1) Let $\mathbf{x}\in G_{i}$ with $i\in \operatorname{PF}(T)$, then $\mathbf{x}\in \operatorname{H}(S)$. If $\mathbf{s}\in S$, then $|\mathbf{x}+\mathbf{s}|=i+|\mathbf{s}| \in T$ since $|\mathbf{s}|\in T$, so $\mathbf{x}+\mathbf{s}\in S$. Conversely, let $\mathbf{x}\in \operatorname{PF}(S)$ and $i=|\mathbf{x}|$. In particular $\mathbf{x}\in G_{i}$ and $i\in \operatorname{H}(T)$. We prove that $i\in \operatorname{PF}(T)$. Let $t\in T \setminus \{0\}$, then $t\mathbf{e}_{j}\in S$ and $\mathbf{x}+t\mathbf{e}_{j}\in S$ for any $j \in \{1, \ldots, d\}$. This means that $|\mathbf{x}+t\mathbf{e}_{j}|\in T$, that is $i+t\in T$.\\
2) From 1) we know that $\mathbf{x}\in \operatorname{PF}(S)$ if and only if $|\mathbf{x}|\in \operatorname{PF}(T)$. So $\mathbf{x}\in \operatorname{SG}(S)\Leftrightarrow 2\mathbf{x}\in S \Leftrightarrow |2\mathbf{x}|\in T \Leftrightarrow 2|\mathbf{x}|\in T \Leftrightarrow |\mathbf{x}|\in \operatorname{SG}(T)$.   
\end{proof}

\begin{proposition}
	Let $T$ be a numerical semigroup and $S\subseteq \mathbb{N}^{d}$ be the $T$-graded GNS. 
	Then $\operatorname{n}(S)=\sum_{i\in \operatorname{N}(T)}|G_i|$.
	\label{trivial}
\end{proposition}
\begin{proof}
	Trivial.
\end{proof}

Observe that if $m\geq 2$ is an integer and $T=\mathbb{N}\setminus \{1,\ldots,m-1\}$ then the $T$-graded and $T$-stripe are the same GNS. 
For GNSs associated to numerical semigroups generated by two elements we can compute the embedding dimension and the type.

\begin{corollary}
Let $T=\langle m,n\rangle$ be a numerical semigroup of embedding dimension 2 and $S\subseteq \mathbb{N}^{d}$ be the $T$-graded GNS. Then:
\begin{itemize}
\item[a)] $\operatorname{e}(S)=\binom{m+d-1}{d-1}+\binom{n+d-1}{d-1}$.
\item[b)] $\operatorname{t}(S)=\binom{mn-m-n+d-1}{d-1}$
\end{itemize}
\end{corollary} 

\begin{proof}
 The first statement easily follows from Theorem~\ref{teorstripe}. Since $\operatorname{PF}(T)=\{F(T)\}$ and $F(T)=mn-m-n$  the second statement follows from Proposition~\ref{pseudoFrobStripe}. 
\end{proof}


If $T=\langle m,n\rangle$ and $S\subseteq \mathbb{N}^{2}$ is the $T$-graded GNS in $\mathbb{N}^{2}$ then the inequality $\operatorname{e}(S)\geq 2\,(\operatorname{t}(S)+1)$ is equivalent to $2mn\leq 3(m+n)-2$ and it is true only for $m=2$ and $n=3$. 
Now we fix $m=2$ and consider the numerical semigroups $T=\langle 2, n\rangle$, with $n>3$ an odd number.

\begin{proposition}
	Let $T=\langle 2,n\rangle$, $n$ an odd integer, $n\geq 5$, and let $S\subseteq \mathbb{N}^{d}$ be the $T$-graded GNS. Then $S$ satisfies the generalized Wilf's conjecture.
\end{proposition}
\begin{proof}
	In such a case we have $\operatorname{H}(T)=\{1,3,5,\ldots,n-2\}$ and, by Proposition \ref{trivial}, $\operatorname{N}(T)=\{0,2,4,n-3\}$. We consider the generalized Wilf's conjecture in its equivalent expression $(\operatorname{e}(S)-d)\operatorname{n}(S)\geq d \operatorname{g}(S)$, that leads to the following inequality: 
	$$\left(\frac{d(d-1)}{2}+\binom{n+d-1}{d-1}\right)\cdot \sum_{k\in\operatorname{H}(T)}\binom{k-1
	+d-1}{d-1}\geq d \cdot \sum_{k\in \operatorname{H}(T)}\binom{k+d-1}{d-1}$$
In particular it suffices to prove the following inequality for all $k\in\{1,3,\ldots,n-2\}$:
$$\binom{n+d-1}{d-1}\cdot\binom{k-1+d-1}{d-1}\geq d \cdot \binom{k+d-1}{d-1}$$
Such an inequality is equivalent to the following:
$$(n+d-1)(n+d-2)\cdots(n+1)\cdot (k-1+d-1)(k-1+d-2)\cdots k\geq d!\cdot (k+d-1)(k+d-2)\cdots (k+1)$$
that reduces to $(n+d-1)(n+d-2)\cdots(n+1)\cdot k \geq d! \cdot (k+d-1)$. Observe that for all $k\in \operatorname{H}(T)$ we have $n+d-1\geq k+d-1$, moreover the inequality $(n+d-2)(n+d-3)\cdots(n+1)\geq d(d-1)\cdots 2$ holds since $n+d-2\geq d$ and $n+d-i\geq d-i$, for $3\leq i\leq d-1$.
\end{proof}

Finally we want to provide another property that a $T$-graded GNS inherits by the numerical semigroup $T$. Let $\leq$ be the natural partial order in $\mathbb{N}^d$. If $S\subseteq \mathbb{N}^d$ is a GNS we define $\operatorname{FA}(S)=\mathrm{Maximals}_{\leq}\operatorname{H}(S)$ and denote $\tau(S)=|\operatorname{FA}(S)|$. In \cite{singhal2021frobenius} the authors define $S$ \emph{quasi-irreducible} if for all $\mathbf{x}\in \operatorname{H}(S)$ then $2\mathbf{x}\in \operatorname{FA}(S)$ or there exists $\mathbf{F}\in \operatorname{FA}(S)$ such that $\mathbf{F}-\mathbf{x}\in S$ and $S$ \emph{quasi-symmetric} if $\tau(S)=\operatorname{t}(S)$. Observe that if $S$ is a $T$-graded GNS then it has in general two or more maximals in $\operatorname{H}(S)$ with respect to $\leq$, so it is never irreducible (\cite{cisto2019irreducible}). We want to describe conditions on $T$ such that the $T$-graded GNS is quasi-irreducible or quasi-symmetric. We first provide the following generalization of \cite[Proposition 2.5]{cisto2019irreducible}.

\begin{proposition}
Let $S\subseteq \mathbb{N}^d$ be a GNS. Then $S$ is quasi-irreducible if and only if $\operatorname{FA}(S)=\operatorname{SG}(S)$.
\label{FA=SG}
\end{proposition}
\begin{proof}
$\Rightarrow)$ Observe that $\operatorname{FA}(S)\subseteq \operatorname{SG}(S)$. Let $\mathbf{x}\in\operatorname{SG}(S)$ and suppose $\mathbf{x} \notin \operatorname{FA}(S)$. By hypotheses there exists $\mathbf{F}\in \operatorname{FA}(S)$ such that $\mathbf{F}-\mathbf{x}\in S\setminus \{\mathbf{0}\}$, in particular there exists $\mathbf{s}\in S\setminus \{\mathbf{0}\}$ such that $\mathbf{x}+\mathbf{s}=\mathbf{F}\notin S$, that is a contradiction.\\
$\Leftarrow)$ Let $\mathbf{x}\in \operatorname{H}(S)$ such that $2\mathbf{x}\notin \operatorname{FA}(S)$. We prove that there exists $\mathbf{F} \in \operatorname{FA}(S)$ such that $\mathbf{F}-\mathbf{x}\in S$. By hypotheses we can assume $\mathbf{x}\notin \operatorname{SG}(S)$, so we have two possibilities:\\
1) $2\mathbf{x}\notin S$ and for all $\mathbf{s}\in S\setminus \{\mathbf{0}\}$ we have $\mathbf{x}+\mathbf{s}\in S$. Observe that, in such a case, for all $i\in \mathbb{N}$ then $i\mathbf{x}+\mathbf{s}\in S$ for all $\mathbf{s}\in S$ and if $k=\max \{i \in \mathbb{N}\mid i\mathbf{x}\notin S\}$ then $k\geq 3$ and $k\mathbf{x}\in \operatorname{SG}(S)=\operatorname{FA}(S)$. Since $2(k-1)>k$ then $2(k-1)\mathbf{x}\in S$, that is $(k-1)\mathbf{x} \in \operatorname{SG}(S)=\operatorname{FA}(S)$, but this is a contradiction since $(k-1)\mathbf{x}\leq k\mathbf{x}$.\\
2) There exists $\mathbf{s}_1\in S\setminus \{\mathbf{0}\}$ such that $\mathbf{x}+\mathbf{s}_1\notin S$. If $\mathbf{x}+\mathbf{s}_1\in \operatorname{SG}(S)$ we have finished. Otherwise put $\mathbf{f}_1=\mathbf{x}+\mathbf{s}_1$ then, arguing as in the proof of (1) of \cite[Proposition 2.6]{cisto2019irreducible} we obtain an element $\mathbf{f}_2\notin S$ such that $\mathbf{f}_2=\mathbf{x}+\mathbf{s}_2$ with $\mathbf{s}_2\in S\setminus \{\mathbf{0}\}$ and $\mathbf{f}_2>\mathbf{f}_{1}$. If $\mathbf{f}_2\in \operatorname{SG}(S)$ we have finished, otherwise by the same argument we obtain a sequence of elements $\mathbf{f}_i\notin S$, for $i>1$, such that $\mathbf{f}_i=\mathbf{x}+\mathbf{s}_i$ with $\mathbf{s}_i\in S\setminus \{\mathbf{0}\}$ and $\mathbf{f}_i>\mathbf{f}_{i-1}$. By the finiteness of $\operatorname{H}(S)$ there exists $k\in \mathbb{N}$ such that $\mathbf{f}_k\in \operatorname{SG}(S)$, that allows to conclude the proof.
\end{proof} 

\begin{corollary}
Let $T$ be a numerical semigroup and $S$ be the $T$-graded GNS in $\mathbb{N}^d$. Then 
\begin{itemize}
\item $S$ is quasi-irreducible if and only if $T$ is irreducible. 
\item $S$ is quasi-symmetric if and only if $T$ is symmetric.
\end{itemize} 
\end{corollary}
\begin{proof}
It is not difficult to see that $\operatorname{FA}(S)=\{\mathbf{a}\in \mathbb{N}^{d}\mid |\mathbf{a}| =\operatorname{F}(T)\}$. So, by Proposition~\ref{pseudoFrobStripe}, it follows that $\operatorname{FA}(S)=\operatorname{SG}(S)$ if and only if $\operatorname{SG}(T)=\{\operatorname{F}(T)\}$, equivalently $T$ is irreducible. Furthermore $\tau(S)=\operatorname{t}(S)$, that is $\operatorname{FA}(S)=\operatorname{PF}(S)$, if and only if $\operatorname{PF}(T)=\{\operatorname{F}(T)\}$, equivalently $T$ is symmetric. 
\end{proof}

\begin{remark}
The analogous of the previous result does not occur for $T$-stripe GNSs. In fact if $T$ is a numerical semigroup with multiplicity $m\geq 4$ and $S$ is the $T$-stripe GNS in $\mathbb{N}^d$ then, $\mathbf{h}=\mathbf{e}_1+(m-3)\mathbf{e}_2\in \operatorname{SG}(S)$ but $\mathbf{h}\notin \operatorname{FA}(S)$, since $\mathbf{e}_1+(m-2)\mathbf{e}_2\in \operatorname{H}(S)$.   
\end{remark}

\section{GNSs having gaps only in the axes}

\begin{definition}
Let $d\in\mathbb{N}$ and $S_1,S_2,\ldots,S_d$ be $d$ numerical semigroups different from $\mathbb{N}$. Set $\mathcal{H}=\bigcup_{i=1}^d\{h\mathbf{e}_i\mid h\in \operatorname{H}(S_i)\}$. It is easy to verify that $S=\mathbb{N}^d \setminus \mathcal{H}$ is a GNS, that we call $\mathrm{Axis}(S_1,S_2,\ldots,S_d)$.
\end{definition}

In order to characterize the minimal generators of $\mathrm{Axis}(S_1,S_2,\ldots,S_d)$ consider the following sets:
\begin{itemize}
\item $F_1=\bigcup_{i=1}^d \{n\mathbf{e}_i\mid n\ \mbox{is a minimal generator of}\ S_i\}$
\item $F_2=\{ \mathbf{e}_i+h\mathbf{e}_j\mid 2\leq h\leq \operatorname{m}(S_j); i,j \in \{1,\ldots,d\}\}, i\neq j\}$.
\item $F_3=\{\mathbf{e}_i+\mathbf{e}_j \mid i,j \in \{1,\ldots,d\}, i<j\}$
\item $F_4=\{\mathbf{e}_i+\mathbf{e}_j+\mathbf{e}_k\mid i,j,k \in \{1,\ldots,d\}, i<j<k\}$
\end{itemize}
Moreover if $d=2$ we assume conventionally that $F_4=\emptyset$ and $\binom{d}{3}=0$.

\begin{proposition}
Let $S$ be the $\mathrm{Axis}(S_1,S_2,\ldots,S_d)$ GNS. Then the set $G=\bigcup_{\ell=1}^4 F_\ell$ is the minimal system of generators of $S$. In particular
$$\operatorname{e}(S)=\sum_{i=1}^d\operatorname{e}(S_i)+(d-1)\sum_{i=1}^d(\operatorname{m}(S_i)-1)+\binom{d}{2}+\binom{d}{3}$$
\end{proposition}
\begin{proof}
Observe that $G\subseteq S$. We first prove that each $\mathbf{s}\in S \setminus (G\cup \{\mathbf{0}\})$ is a sum of elements in $G$. If $\mathbf{s}=\lambda\mathbf{e}_i$ for some $i\in \{1,\ldots,d\}$, that is $\mathbf{s}$ belongs to $i$-th axis, then it is not difficult to check that $\mathbf{s}$ is generated by the elements in $F_1$. Suppose $\mathbf{s}=\mathbf{e}_i+\lambda \mathbf{e}_j$, with $\lambda\neq 0$ and $i\neq j$. If $1\leq \lambda\leq \operatorname{m}(S_j)$ then $\mathbf{s}\in G$, so we suppose $\lambda> \operatorname{m}(S_j)$. Let $k=\max \{n\mid n\operatorname{m}(S_j)<\lambda\}$, then $\mathbf{s}=\mathbf{e}_i+\lambda \mathbf{e}_j=\mathbf{e}_i+(\lambda-k\operatorname{m}(S_j))\mathbf{e}_j+k\operatorname{m}(S_j)\mathbf{e}_j$, where $\mathbf{e}_i+(\lambda-k\operatorname{m}(S_j))\mathbf{e}_j\in F_2\cup F_3$ and $k\operatorname{m}(S_j)\mathbf{e}_j$ is generated by the set $F_1$. So $\mathbf{e}_i+\lambda \mathbf{e}_j$ is a sum of elements in $G$ for all $\lambda\in \mathbb{N}$ and for all $i,j\in\{1,\ldots,d\}$, $i\neq j$. Suppose that $\mathbf{s}=\gamma\mathbf{e}_i+\lambda \mathbf{e}_j$, $\gamma,\lambda \in \mathbb{N}\setminus\{0\}$. If $\gamma=1$ or $\lambda=1$ we are done, otherwise we can write $\mathbf{s}=\mathbf{e}_j+(\gamma-1)\mathbf{e}_i+\mathbf{e}_i+(\lambda-1) \mathbf{e}_j$. It is easy to check that it is a sum of elements in $G$ by the previous argument. So for all $\gamma,\lambda \in \mathbb{N}\setminus\{0\}$ we obtain $\mathbf{s}=\gamma\mathbf{e}_i+\lambda \mathbf{e}_j$ as a sum of elements in $G$. Suppose $\mathbf{s}=\alpha\mathbf{e}_i+\beta\mathbf{e}_j+\gamma\mathbf{e}_k$ with $\alpha,\beta,\gamma\in \mathbb{N}\setminus\{0\}$ and $i<j<k$. If $\alpha=\beta=\gamma=1$ then $\mathbf{s}\in F_4$, otherwise we can suppose without loss of generality that $\alpha>1$ and, in such a case, $\mathbf{s}=[(\alpha-1)\mathbf{e}_i+\beta\mathbf{e}_j]+[\mathbf{e}_i+\gamma\mathbf{e}_k]$, that is a sum of elements in $G$. So, $\alpha\mathbf{e}_i+\beta\mathbf{e}_j+\gamma\mathbf{e}_k$ is a sum of elements in $G$ for all $\alpha,\beta,\gamma\in \mathbb{N}\setminus\{0\}$ and for all $i<j<k$. Finally consider $\mathbf{s}=\sum_{i=1}^r\lambda_i\mathbf{e}_{k_i}$, $r>3$, $\{k_1,\ldots,k_r\}\subseteq \{1,\ldots,d\}$ and $\lambda_i\in \mathbb{N}\setminus \{0\}$ for all $i\in \{1,\ldots,r\}$. In such a case, if $r$ is even we consider $\mathbf{s}=\sum_{i=1}^{r/2} (\lambda_{2i-1}\mathbf{e}_{k_{2i-1}}+\lambda_{2i}\mathbf{e}_{k_{2i}})$, if $r$ is odd we consider $\mathbf{s}=\sum_{i=1}^{(r-3)/2} (\lambda_{2i-1}\mathbf{e}_{k_{2i-1}}+\lambda_{2i}\mathbf{e}_{k_{2i}})+ (\lambda_{r-2}\mathbf{e}_{k_{r-2}}+\lambda_{r-1}\mathbf{e}_{k_{r-1}}+\lambda_{r}\mathbf{e}_{k_{r}})$. In both cases it is easy to argue that $\mathbf{s}$ is a sum of elements in $G$. So each $\mathbf{s}\in S \setminus (G\cup \{\mathbf{0}\})$ is sum of elements in $G$, in particular all generators of $S$ are contained in $G$. Moreover all elements in $G$ cannot be expressed as a sum of non zero elements of $S$, so every element of $G$ is a minimal generator of $S$.  
\end{proof}

\begin{proposition}
Let $S$ be the $\mathrm{Axis}(S_1,S_2,\ldots,S_d)$ GNS. Then $\operatorname{PF}(S)=\bigcup_{i=1}^d\{f\mathbf{e}_i\mid f\in\operatorname{PF}(S_i) \}$. In particular $\operatorname{t}(S)=\sum_{i=1}^d\operatorname{t}(S_i)$.
\end{proposition}
\begin{proof}
Trivial. 
\end{proof}

\begin{remark}
Let $S$ be the $\mathrm{Axis}(S_1,S_2,\ldots,S_d)$ GNS. Then:
\begin{enumerate}
\item $\operatorname{n}(S)=\sum_{i=1}^d \operatorname{n}(S_i)-(d-1)$
\item $\operatorname{c}(S)=\sum_{i=1}^d \operatorname{c}(S_i)-(d-1)$
\item $\operatorname{t}(S)=\sum_{i=1}^d\operatorname{t}(S_i)$
\end{enumerate}
In fact all gaps belong to the coordinate axes of $\mathbb{N}^d$, the term $(d-1)$ occurs since $\mathbf{0}$ belongs to all coordinate axes and obviously $\operatorname{PF}(S)=\bigcup_{i=1}^d\{f\mathbf{e}_i\mid f\in\operatorname{PF}(S_i) \}$.
\label{invariants}
\end{remark}

\begin{proposition}
Let $S$ be the $\mathrm{Axis}(S_1,S_2,\ldots,S_d)$ GNS and suppose that $S_i$ satisfies Wilf's conjecture, that is $\operatorname{e}(S_i)\operatorname{n}(S_i)\geq \operatorname{c}(S_i)$, for all $i\in \{1,\ldots,d\}$. Then $S$ satisfies the generalized Wilf's conjecture.
\end{proposition}
\begin{proof}
Considering the previous results, by a direct computation we obtain: \begin{flalign*}
& \operatorname{e}(S)\operatorname{n}(S)  =\left(\sum_{i=1}^d\operatorname{e}(S_i)\right)\left(\sum_{j=1}^d\operatorname{n}(S_j)-(d-1)\right)+  &\\
&+\left(\sum_{j=1}^d\operatorname{n}(S_i)-(d-1)\right)\left[(d-1)\sum_{i=1}^d(\operatorname{m}(S_i)-1)+\binom{d}{2}+\binom{d}{3}\right]=  &\\
&=\sum_{i=1}^d\operatorname{e}(S_i)\operatorname{n}(S_i)+ (d-1)\sum_{i=1}^d\operatorname{n}(S_i)\operatorname{m}(S_i)+\sum_{i\neq j}\operatorname{e}(S_i)\operatorname{n}(S_j)+ (d-1)\sum_{i\neq j}\operatorname{n}(S_j)\operatorname{m}(S_i)+&\\
&+\left(\sum_{i=1}^d\operatorname{n}(S_i)-(d-1)\right)\left(\binom{d}{2}+\binom{d}{3}-d(d-1)\right)-(d-1)^2\sum_{i=1}^d\operatorname{m}(S_i)-(d-1)\sum_{i=1}^d\operatorname{e}(S_i)&
\end{flalign*}
Since for all $i\in\{1,\ldots,d\}$ we have $\operatorname{m}(S_i)\geq \operatorname{e}(S_i)$, $\operatorname{n}(S_i)\geq 1$, by hypotheses $\operatorname{e}(S_i)\operatorname{n}(S_i)\geq \operatorname{c}(S_i)$, (2) of Remark~\ref{invariants}, we can continue the argument:
\begin{flalign*}
& \operatorname{e}(S)\operatorname{n}(S) \geq &\\
&\geq d \sum_{i=1}^d\operatorname{c}(S_i)+d\sum_{i\neq j}\operatorname{e}(S_i)\operatorname{n}(S_j)+\left(\binom{d}{2}+\binom{d}{3}-d(d-1)\right)-(d^2-d)\sum_{i=1}^d\operatorname{e}(S_i)=&\\
 & = d\operatorname{c}(S)+d(d-1)+ d\left(\operatorname{e}(S_1)\sum_{j\neq 1}^d\operatorname{n}(S_j)+ \cdots +\operatorname{e}(S_d)\sum_{j\neq d}^d\operatorname{n}(S_j) \right)+&\\
 &\left(\binom{d}{2}+\binom{d}{3}-d(d-1)\right)-(d^2-d)\sum_{i=1}^d\operatorname{e}(S_i)\geq&\\
 &\geq d\operatorname{c}(S)+d(d-1)\sum_{i=1}^d\operatorname{e}(S_i)+\binom{d}{2}+\binom{d}{3}-(d^2-d)\sum_{i=1}^d\operatorname{e}(S_i) =&\\
 &=d\operatorname{c}(S)+\binom{d}{2}+\binom{d}{3}>d\operatorname{c}(S).&
\end{flalign*}
\end{proof}

As for $T$-graded GNSs, we can prove:

\begin{proposition}
Let $S$ be the $\mathrm{Axis}(S_1,S_2,\ldots,S_d)$ GNS, then the following hold:
\begin{enumerate}
\item $S$ is quasi-irreducible if and only if $S_i$ is irreducible for all $i\in \{1,\ldots,d\}$.
\item $S$ is quasi-symmetric if and only if $S_i$ is symmetric for all $i\in \{1,\ldots,d\}$.
\end{enumerate}
\label{axisQuasi}
\end{proposition}
\begin{proof}
The claim (1) easily follows from Proposition~\ref{FA=SG}, since $\operatorname{SG}(S_i)=\{\operatorname{F}(S_i)\}$ for all $i\in\{1,\ldots,d\}$ if and only if $\operatorname{FA}(S)=\operatorname{SG}(S)=\{\operatorname{F}(S_i)\mathbf{e}_i\mid i\in\{1,\ldots,d\}\}$. The claim (2) follows since $\operatorname{PF}(S_i)=\{\operatorname{F}(S_i)\}$ for all $i\in\{1,\ldots,d\}$ if and only if $\operatorname{FA}(S)=\operatorname{PF}(S)=\{\operatorname{F}(S_i)\mathbf{e}_i\mid i\in\{1,\ldots,d\}\}$.
\end{proof}

\begin{remark}
It is known (\cite{singhal2021frobenius}) that if $S\subseteq \mathbb{N}^d$ is a quasi-irreducible GNS then $\tau(S)\leq \operatorname{t}(S)\leq 2\tau(S)$. We can observe that it is possible to produce a GNS such that $\operatorname{t}(S)$ is any desired value between $\tau(S)$ and $2\tau(S)$. In fact, let $r\in \mathbb{N}$ such that $d\leq d+r\leq 2d$ and consider $S_1,\ldots,S_r$ pseudo-symmetric numerical semigroups (in particular $\operatorname{PF}(S_i)=\{\operatorname{F}(S_i),\operatorname{F}(S_i)/2\}$ for all $i\in \{1,\ldots,r\}$), and $S_{r+1},\ldots, S_d$ symmetric numerical semigroups (in particular $\operatorname{PF}(S_j)=\{\operatorname{F}(S_j)\}$ for all $j\in \{r+1,\ldots,d\}$). So the semigroup $S=\mathrm{Axis}(S_1,S_2,\ldots,S_d)$ is quasi-irreducible with $\tau(S)=d$ and $\operatorname{t}(S)=\tau(S)+r$.
\end{remark}

\section{Concluding remarks}
We conclude with some questions and possible developments arising from this paper:
\begin{itemize}

\item In this paper we consider, for a GNS $S$, the inequality $\operatorname{e}(S)\geq d(\operatorname{t}(S)+1)$ and we provide a class of GNSs whose elements satisfy such an inequality. It could be interesting to find other different classes of GNSs that satisfy the inequality.

\item For a $T$-graded GNS we prove the generalized Wilf's conjecture in the case $T=\langle m,n\rangle$, $n\geq 3$ odd integer. In general it seems very difficult to prove the conjecture for $T=\langle m,n\rangle$ with $m>2$ and $n>m$. So generalized Wilf's conjecture for such a class remains open.

\item  We describe how some invariants and properties of a $T$-stripe and a $T$-graded GNS are related to the invariants and properties of the numerical semigroup $T$. We ask if there are other properties of a $T$-stripe or a $T$-graded GNS related to the properties of the associated numerical semigroup $T$. The same question can be considered for the GNS $\mathrm{Axis}(S_1,S_2,\ldots,S_d)$, with respect to the related numerical semigroups $S_1,\ldots,S_d$. 

\end{itemize}


Finally we mention that in order to prove the generalized Wilf's conjecture for every GNS $\mathrm{Axis}(S_1,S_2,\ldots,S_d)$ and $d\geq 2$, we attempt also to use Corollary~\ref{wilfType}. We observe that if $S$ is the $\mathrm{Axis}(S_1,S_2,\ldots,S_d)$ GNS and if $\operatorname{e}(S_i)+\operatorname{m}(S_i)\geq 2\operatorname{t}(S_i)+2$ for all $i\in \{1,\ldots,d\}$, then $S$ satisfies the generalized Wilf's conjecture. In fact we have that $\operatorname{t}(S_i)\leq \operatorname{m}(S_i)-1$ for all $i\in \{1,\ldots,d\}$. Moreover in such a case  we can suppose $\operatorname{t}(S_j)< \operatorname{m}(S_j)-1$ for some $j$, since if $\operatorname{t}(S_i)= \operatorname{m}(S_i)-1$ for all $i$ then each $S_i$ has maximal embedding dimension and satisfies Wilf's conjecture. So we obtain:
\begin{flalign*}
&\operatorname{e}(S)=\sum_{i=1}^d\operatorname{e}(S_i)+(d-1)\sum_{i=1}^d(\operatorname{m}(S_i)-1)+\binom{d}{2}+\binom{d}{3} =&\\
&= \sum_{i=1}^d(\operatorname{e}(S_i)+\operatorname{m}(S_i)-1)+(d-2)\sum_{i=1}^d(\operatorname{m}(S_i)-1)+\binom{d}{2}+\binom{d}{3}\geq &\\
&\geq \sum_{i=1}^d (2\operatorname{t}(S_i)+1)+(d-2)\sum_{i=1}^d\operatorname{t}(S_i)+\binom{d}{2}+\binom{d}{3}= d\sum_{i=1}^d\operatorname{t}(S_i)+d+\binom{d}{2}+\binom{d}{3}> &\\
&>  d\sum_{i=1}^d\operatorname{t}(S_i)+d=d\left(\sum_{i=1}^d\operatorname{t}(S_i)+1\right)=d(\operatorname{t}(S)+1)&
\end{flalign*}
We have checked by brute force, using the \texttt{GAP} (\cite{GAP}) package \texttt{numericalsgps} (\cite{numericalsgps}), that every numerical semigroup $S$ of genus $g\leq 26$ satisfies $\operatorname{e}(S_i)+\operatorname{m}(S_i)\geq 2\operatorname{t}(S_i)+2$. But this is not true in general. 
In fact, in a personal communication, Shalom Eliahou provided us different numerical semigroups of higher genus not verifying the previous inequality. For instance, if $S=<17, 18, 20, 22, 24, 25>$, then $\operatorname{g}(S)=27$, $m(S)=17$, $e(S)=6$ and $t(S)=11$, hence $m(S)+e(S)-2(t(S)+1)=-1$. He suggested also the following interesting question: \\

\textbf{Question}:\\
Let $m$ be a positive integer and $\cM(m)$ be the set of all numerical semigroups of multiplicity $m$. Set
$$g(m) = \inf_{S\in \cM(m)} \{m+e(S)-2(t(S)+1)\}$$
so $g(m)$ belongs to $\ZZ \cup \{-\infty\}$. From some computational tests, it is verified that $g(m)\leq 0$ for $2\leq m\leq 16$, $g(17)\leq -1$, $g(18)\leq -1$, $g(19)\leq -3$. We do not know if these bounds are sharp.  \\
 We ask if $g(m)$ is always an integer. In such a case, it would be very interesting to determine the behavior of $g(m)$ as a function of $m$.


\medskip

\textbf{Acknowledge} The authors wish to thank Professor Shalom Elihaou for his suggestions related to the last part of this work, and for his nice availability and kindness. They would like to thank also Professor Rosanna Utano for her helpful suggestions and comments.



\bibliographystyle{plain}
\bibliography{stripe}

\end{document}